\documentclass[amssymb,amsfonts,refcheck,12pt,verbatim,righttag]{amsart}

\usepackage{color}
\usepackage{amssymb}
\usepackage{graphics}
\usepackage{tikz}
\usepackage{multicol}
\usepackage{subfigure}
\usepackage{bm}
\usepackage{algorithm}
\usepackage{algpseudocode}

\def\colon{{:}\;}

 \numberwithin{equation}{section} % \renewcommand{\rm}{\normalshape} %
\newcommand{\beq}{\begin{equation}}
\newcommand{\eeq}{\end{equation}}

\newcommand{\ben}{\begin{eqnarray}}
\newcommand{\een}{\end{eqnarray}}

\newcommand{\bet}{\begin{eqnarray*}}
\newcommand{\eet}{\end{eqnarray*}}

\newtheorem{thm}{Theorem}[section]
\newtheorem{lem}[thm]{Lemma}

\newtheorem{cor}[thm]{Corollary}

\newtheorem{rem}[thm]{Remark}

\usepackage[colorlinks,citecolor=red,pagebackref,hypertexnames=false]{hyperref}

\newcommand{\R}{\mathbb{R}}

\newcommand{\Q}{\mathbb{Q}}
\newcommand{\N}{\mathbb{N}}
\newcommand{\Z}{\mathbb{Z}}

\theoremstyle{plain}

\begin{document}
\baselineskip 16pt

\title{Dimension drop for intersections of Cantor sets}

\author{Lai Jiang}
\address[Lai Jiang]{School of Fundamental Physics and Mathematical Sciences, Hangzhou Institute for Advanced Study, UCAS, Hangzhou 310024, China}
\email{jianglai@ucas.ac.cn}

\author{Bing Li}
\address[Bing Li]{School of Mathematics, South China University of Technology, Guangzhou 510641, China}
\email{scbingli@scut.edu.cn}

\author{Ruofan Li}
\address[Ruofan Li]{Department of Mathematics, Jinan University, Guangzhou 510632, China}
\email{liruofan@jnu.edu.cn}

\author{Yufeng Wu*}
\thanks{*Corresponding author}
\address[Yufeng Wu]{School of Mathematics and Statistics\\HNP-LAMA\\Central South University\\ Changsha 410083, China}
\email{yufengwu.wu@csu.edu.cn}

\keywords{self-similar sets, missing-digit sets, fractal intersections, Minkowski dimension, Hausdorff dimension}
\thanks{2010 \textit{Mathematics Subject Classification}: 28A80}

\begin{abstract}
Let $E\subset \R$ be a self-similar set generated by a homogeneous iterated function system $\Phi$ with contraction ratio $\rho\in (0,1)$. Assume that $\Phi$ satisfies the open set condition and $\dim_{\rm H}E<1$.   Let $f$ be a $C^1$-diffeomorphism on $\R$. We prove that if $\log|f'(x)|/\log\rho\not\in\Q$ for every $x\in E\cap f^{-1}(E)$, then the upper Minkowski dimension of $f(E)\cap E$
is strictly less than the Hausdorff dimension of $E$. We also establish a quantitative dimension drop result when $E$ is a missing-digit set and  $f$ is an affine map with rational slope satisfying a certain arithmetic condition.     
Based on these results and a result of Shmerkin~\cite{Shmerkin19}, we obtain characterizations of $\gamma$ in various contexts such that $\overline{\dim}_{\rm M}\left((\gamma E+\alpha)\cap E\right)<\dim_{\rm H}E$ for every $\alpha\in\R$.
\end{abstract}

\maketitle

\section{Introduction}

The study of intersections of fractal sets is an important topic in fractal geometry,  closely related to other topics such as slices, affine embeddings, and Lipschitz equivalence. Given a self-similar set $E\subset \R$ and a map $f: \R\to \R$, a general question is under what assumptions on $E$ and $f$ does the following  dimension drop property hold:  the Hausdorff dimension of the intersection $f(E)\cap E$ is strictly less than that of $E$. It is natural to expect this holds under some arithmetic non-resonance conditions. 
In this paper, we study this question from both qualitative and quantitative aspects.

We first recall some terminologies about self-similar sets. We call a finite collection $\Phi=\{\phi_i\}_{i=1}^{m}$ of contracting similitudes on $\R^d$ a (self-similar) {\em iterated function system} (IFS). According to a classical result of Hutchinson \cite{hutchinson}, given an IFS $\Phi=\{\phi_i\}_{i=1}^{m}$ on $\R^d$, there is a unique non-empty compact set $E\subset\R^d$, called the {\em self-similar set} generated by $\Phi$, such that
\begin{equation}\label{deSSS}
E=\bigcup_{i=1}^{m}\phi_i(E).
\end{equation}
To avoid triviality, we always assume $m\geq 2$. 
If all $\phi_i$ have the same linear part, we call $\Phi$ a {\em homogeneous} IFS and $E$ a homogeneous self-similar set. We say that $\Phi$ (or $E$) satisfies the {\em open set condition} (OSC) if there exists a non-empty bounded open set $V\subset \R^d$ such that $\phi_i(V)$ ($1\leq i\leq m$) are pairwise disjoint and contained in $V$.
Under the OSC,  the Hausdorff dimension of $E$ is given by the {\em similarity dimension} of $\Phi$, that is, the positive number $s$ satisfying $\sum_{i=1}^{m}\rho_i^s=1$, where for $1\leq i\leq m$,  $\rho_i\in (0,1)$ is the contraction ratio of $\phi_i$. 
 We say that $\Phi$ satisfies the {\em strong separation condition} (SSC) if \eqref{deSSS} is a disjoint union. It is well-known that the SSC implies the OSC. 
For proofs of the above facts and  more information on self-similar sets, see~\cite{Falconer14}.

In this paper, we are concerned with the intersection $f(E)\cap E$,
where $E\subset \R$ is a  self-similar set generated by a homogeneous IFS $\Phi$ satisfying the OSC and $f:\R\to \R$ is a $C^1$-diffeomorphism or an affine map. 

Much work has been devoted to the size of $f(E)\cap E$ when $f$ is affine. 
Let $\gamma,\alpha\in \R$ with $\gamma\neq 0$. The measure of the  intersection
\begin{equation}\label{eqKbDint}
(\gamma E+\alpha)\cap E
\end{equation}
has been investigated in \cite{EKM10S}. Let $\mu$ be a self-similar measure on $E$. According to \cite[Theorem 4.5]{EKM10S}, if $E$ satisfies SSC, then $(\gamma E+\alpha)\cap E$ has positive $\mu$-measure only when $\gamma E+\alpha$ contains an elementary piece of $E$ (i.e., a set of the form $\phi_{i_1}\circ\cdots\circ\phi_{i_k}(E)$ for some word $i_1\ldots i_k$ on $\{1,\ldots,m\}$). Then by the Logarithmic Commensurability Theorem of Feng and Wang~\cite{FengWang09O} (see Theorem~\ref{thmFW}), this happens only when  
\[\frac{\log|\gamma|}{\log \rho}\in \Q,\]
where $\rho\in (0,1)$ is the common contraction ratio of the elements in $\Phi$. 
Hence, if $E$ satisfies the SSC, then for any $\gamma\in \R\setminus\{0\}$ with $\log |\gamma|/\log \rho\not\in \Q$, we have the following measure drop property:
\begin{equation}\label{eqmu0}
	\mu((\gamma E+\alpha)\cap E)=0, \quad \forall \alpha\in \R.
\end{equation}
A natural question is: under the same condition on $\gamma$, do  we have a stronger dimension drop property, asserting that the Hausdorff dimension of  $(\gamma E+\alpha)\cap E$ is  strictly less than that of $E$? Note that this does not follow from \eqref{eqmu0}.

In this paper, our first result answers the above question in the affirmative. Indeed, we prove a much more general result for $C^1$-intersections of $E$ when $E$ satisfies the OSC and has Hausdorff dimension less than $1$,  which states for any $C^1$-diffeomorphism $f$ on $\R$ whose derivative satisfies a natural arithmetic condition, the upper Minkowski dimension of $f(E)\cap E$ is strictly less than the Hausdorff dimension of $E$. 

To state our  results, we need to introduce some more notation.  For any non-empty bounded set $A\subset \R$, the upper Minkowski dimension of $A$, denoted by $\overline{\dim}_{\rm M}A$, is defined by 
\[\overline{\dim}_{\rm M}A=\limsup_{\delta\to0}\frac{\log N_{\delta}(A)}{-\log\delta},\]
where for $\delta>0$, $N_{\delta}(A)$ denotes the smallest number of closed intervals of length $\delta$ needed to cover $A$. Let $\dim_{\rm H}A$ denote the Hausdorff dimension of $A$. We adopt the convention that $\overline{\dim}_{\rm M}A=\dim_{\rm H}A=0$ when $A$ is the empty set.  For more information on the Minkowski and Hausdorff dimensions, we refer the reader to~\cite{Falconer14}.

Our first main result in this paper is the following. 

\begin{thm}\label{thmC1drop}
	Let \(E\subset\R\) be the self-similar set  generated by a homogeneous IFS
	$\Phi=\{\phi_i(x)=\rho x+a_i\}_{i=1}^m$
	satisfying the OSC, where $0<\rho<1$ and $a_i\in\R$, $i=1,\ldots, m$. Assume $\dim_{\rm H} E<1$.
	Let \(f\) be a \(C^1\)-diffeomorphism on $\R$. Suppose that
	\[\frac{\log |f'(x)|}{\log \rho}\notin\Q
	\qquad\text{ for every }x\in E\cap f^{-1}(E).
	\]
	Then
	\[\overline{\dim}_{\rm M}(f(E)\cap E)<\dim_{\rm H}E.
	\]
\end{thm}

Theorem~\ref{thmC1drop} is closely related to a result of Feng, Huang, and Rao~\cite{FWR14A} on $C^1$-embeddings of self-similar sets.  Let $E,F\subset \R^d$ be self-similar sets generated by IFSs $\Phi$ and $\Psi$, respectively. Assume that $E$ satisfies the OSC and  $\dim_{\rm H}F$ equals the similarity dimension of $\Psi$. It was proved in~\cite[Theorem 1.1]{FWR14A} that $F$ can be $C^1$-embedded into $E$ if and only if $F$ can be affinely embedded into $E$. Furthermore, if $F$ cannot be affinely embedded into $E$, then 
\[\dim_{\rm H}\left(E\cap f(F)\right)<\dim_{\rm H}F\]
for any $C^1$-diffeomorphism on $\R^d$. Clearly, the above result of Feng, Huang, and Rao does not apply to the case $E=F$, as in this case $F$ certainly can be affinely embedded into $E$ by the identity map. However, Theorem~\ref{thmC1drop} says that for the single-set case, we do have dimension drop for $C^1$-intersections whenever the derivative of $f$ satisfies a natural arithmetic condition. We prove Theorem~\ref{thmC1drop} by combining an adaptation of the localization and recalling argument in \cite{FWR14A} and the Logarithmic Commensurability Theorem of Feng and Wang~\cite{FengWang09O}. The new point is to
keep track of the slopes of the affine maps produced in the recalling process; see Lemma~\ref{lemBlowup}. 

We also note that Hochman~\cite{Hochman18S} studied a related but different problem. Among other results, he proved that if $X\subset [0,1]$ is a $\times a$-invariant set satisfying certain assumptions and $f$ is a $C^1$-diffeomorphism on $\R$ such that $f(X)\subset X$, then $\log|f'(x)|/\log a\in\Q$ for all $x\in X$. Here $a\geq 2$ is an integer and $X$ is $\times a$-invariant means $ax\pmod1\in X$ whenever $x\in X$.  Hochman's result is different from Theorem~\ref{thmC1drop}, as the latter concerns dimension drop for $C^1$-intersections rather than $C^1$-embeddings of self-similar sets. Certainly Theorem~\ref{thmC1drop} implies that if $f(E)\subset E$ then $\log|f'(x)|/\log \rho\in\Q$ for at least one point $x\in E\cap f^{-1}(E)$.

When applied to affine maps, Theorem~\ref{thmC1drop} has the following consequence,  which characterizes for a class of self-similar sets $E$, those positive rational slopes $\gamma$ such that dimension drop occurs for the intersection $(\gamma E+\alpha)\cap E$ for every $\alpha\in\R$.
\begin{cor}\label{CorEGamma}
	Let $E$ be as in Theorem~\ref{thmC1drop} and assume in addition that $\rho^{-1}$ is a prime number.  Let $\gamma\in \Q$ with $\gamma>0$. Then 
\begin{equation}\label{eqgEaEless}
\overline{\dim}_{\rm M}\left((\gamma E+\alpha)\cap E\right)<\dim_{\rm H}E
\end{equation}
for every $\alpha\in \R$ if and only if $\gamma\not\in \{\rho^{k}: k\in\Z\}$.
\end{cor}

\begin{rem}\label{remps}
{\rm (i)} The ``if'' part  may fail without assuming $\rho^{-1}$ is prime. For instance, let  $\rho^{-1}=9$ and let $E$ be generated by the IFS $\{\rho (x+i):i=0,2,6,8\}$. Take $\gamma=3$. Then $\gamma\not\in \{\rho^k: k\in\Z\}$. However, note that $E$ coincides with the middle-third Cantor set $C$ and  $3C\cap C=C$. Hence the ``if'' part fails for this example.

{\rm (ii)} We say that $E$ is  symmetric if $E=c-E$ for some $c\in\R$. Let $E$ and $\rho$ be as in Corollary~\ref{CorEGamma} and assume that $E$ is symmetric. Then a slight modification of the proof of  Corollary~\ref{CorEGamma} yields the following: Let $\gamma\in \Q$ with $\gamma<0$. Then \eqref{eqgEaEless} holds for every $\alpha\in\R$ if and only if $\gamma\not\in\{-\rho^k: k\in\Z\}$. 
\end{rem}

Theorem~\ref{thmC1drop} is purely qualitative, since the proof is based on contradiction. 
We would like to also explore quantitative dimension drop results. For this purpose, we restrict our attention to affine intersections of missing-digit sets, a subclass of homogeneous self-similar sets on $\R$ satisfying the OSC. 

Let $b\geq 3$ be an integer and let $D$  be a subset of $\{0,1,\ldots, b-1\}$ with cardinality  $2\leq \#D<b$. The \textit{missing-digit set} $K_{b,D}$ with base $b$ and digit set $D$ is  defined by
\[K_{b,D}=\left\{\sum_{j=1}^{\infty}\frac{\varepsilon_j}{b^j}: \varepsilon_j\in D \text{ for all } j\geq 1\right\}.\]
Equivalently, $K_{b,D}$ is the self-similar set satisfying 
\begin{equation}\label{eqKbDiden}
	K_{b,D}=\bigcup_{d\in D}\left(\frac{1}{b}K_{b,D}+\frac{d}{b}\right).
\end{equation}

Shmerkin~\cite[Corollary 6.4]{Shmerkin19} proved that 
if $\gamma$ is irrational, then 
\begin{equation}\label{eqShmerkin}
\overline{\dim}_{\rm M}\left((\gamma K_{b,D}+\alpha)\cap K_{b,D}\right)\leq \max\{2\dim_{\rm H}K_{b,D}-1, 0\}
\end{equation}
for any $\alpha\in \R$. Since $\dim_{\rm H}K_{b,D}=\log(\#D)/\log b<1$,   
this shows that, when $\gamma$ is irrational, we  have the dimension drop property:
\begin{equation}\label{eqdimdropinq}
\overline{\dim}_{\rm M}\left((\gamma K_{b,D}+\alpha)\cap K_{b,D}\right)< \dim_{\rm H}K_{b,D}, \quad \forall \alpha\in \R.
\end{equation}
It remains to consider rational $\gamma$. According to Theorem~\ref{thmC1drop}, if $\gamma\in \Q\setminus\{0\}$ satisfies $\log|\gamma|/\log b\not\in\Q$, then \eqref{eqdimdropinq} still holds. In view of \eqref{eqShmerkin},  a natural question is whether for such $\gamma$ we can obtain quantitative strengthening of \eqref{eqdimdropinq}. We first point out that, in general, we cannot expect Shmerkin's bound \eqref{eqShmerkin} to hold for these $\gamma$. To see this, let $C=K_{3,\{0,2\}}$ be the middle-third Cantor set and take $\gamma=4$. Then $\dim_{\rm H}C=\frac{\log2}{\log3}$ and $\log\gamma/\log b=\log 4/\log3\not\in\Q$. However, it is easily checked that $4C\cap C$ contains the self-similar set generated by the IFS $\left\{\frac{x}{9}, \frac{x+8}{9}\right\}$, which has Hausdorff dimension $\frac{\log2}{\log9}$. Hence,
\[\overline{\dim}_{\rm M}(4C\cap C)\geq\dim_{\rm H}(4C\cap C)\geq \frac{\log 2}{\log 9}>\max\left\{\frac{2\log 2}{\log 3}-1,0\right\}.\]
Thus the analogue of~\eqref{eqShmerkin} fails for this example. Nevertheless, in the following,  we establish a quantitative dimension drop result for a class of rationals $\gamma$ satisfying a certain arithmetic condition. 

\begin{thm}\label{thmdimensiondrop}
Let $\gamma=\frac{s}{r}b^{\ell}$ be a rational number, where $r\in\N$, $\ell\in\Z $, $s\in\Z\setminus\{0\}$ are such that  $|s|\neq r$ and $\gcd(s,r)=\gcd(sr,b)=1$. Let  $q\in \N$ with   $q>\frac{\log(|s|+r)}{\log(b/(\#D))}+1$.  Then 
\begin{equation}\label{eqKbDalphaupp}
\overline{\dim}_{\rm M} \left((\gamma K_{b,D}+\alpha) \cap K_{b,D}\right)\leq \frac{\log((\#D)^q-1)}{q\log b}
\end{equation}
for any $\alpha\in \R$.
\end{thm}

\begin{rem}\label{Rem1Thm1}
{\rm (i)} Since $\dim_{\rm H}K_{b,D}=\log(\#D)/\log b$, \eqref{eqKbDalphaupp}  implies that   
\begin{equation}\label{eqstrinq}
\overline{\dim}_{\rm M}\left( (\gamma K_{b,D}+\alpha) \cap K_{b,D}\right)<\dim_{\rm H}K_{b,D}, \qquad \forall \alpha\in\R.
\end{equation}
	
{\rm (ii)} Note that the upper bound in \eqref{eqKbDalphaupp} is independent of the exponent $\ell$ and the translation $\alpha$.
\end{rem}

Under the assumption that $b$ is a prime and $K_{b,D}$ is symmetric (one such example is the middle-third Cantor set), we can combine Corollary~\ref{CorEGamma}, Remark~\ref{remps}(ii) and Shmerkin's result \cite[Corollary 6.4]{Shmerkin19} to  give a complete characterization of  $\gamma\in \R$ such that \eqref{eqstrinq} holds for every $\alpha\in\R$. This strengthens Corollary~\ref{CorEGamma} for a class of missing-digit sets.	

\begin{cor}\label{CorChar}
Suppose that $b$ is a prime number and $K_{b,D}$ is symmetric. 
Let $\gamma\in\R$. Then \eqref{eqstrinq} holds for every $\alpha\in\R$  if and only if $\gamma\notin\{\pm b^{k}: k\in \Z\}$.
\end{cor}

We now summarize some related studies in the literature. Intersections of self-similar sets with their affine copies  fit naturally  into the more general framework of slices of fractal sets, as the intersection $(\gamma E+\alpha)\cap E$
can be identified as the slice $\ell_{\gamma, \alpha}\cap (E\times E)$
through a non-singular affine transformation, where for $\gamma, \alpha\in\R$ with $\gamma\neq 0$, $\ell_{\gamma, \alpha}$ denotes the line in $\R^2$ with slope $\gamma$ that passes through  $(0,\alpha)$. 

Marstrand’s Slicing Theorem~\cite{Marstrand54S} implies that, for every Borel set $F\subset \R^2$ and every fixed slope $\gamma$,   
\begin{equation}\label{eqMarstrand}
\dim_{\rm H} (\ell_{\gamma, \alpha}\cap F)\leq \max\{\dim_{\rm H}F-1,0\}
\end{equation}
for Lebesgue almost all $\alpha\in\R$. A higher-dimensional generalization was given by Mattila~\cite{Mattila95}.  
Although \eqref{eqMarstrand} gives a dimension bound for a typical slice of $F$, it is in general challenging to understand the intersection of $F$ with a fixed line. A famous conjecture of Furstenberg~\cite{Furstenberg70} says that if $F$ is the Cartesian product of a $\times p$-invariant set and a $\times q$-invariant set, where $p,q\geq 2$ are integers with $\log p/\log q\notin \Q$,  then \eqref{eqMarstrand} holds for all lines not parallel to the axes. Furstenberg's conjecture was proved independently by Shmerkin~\cite{Shmerkin19} and Wu~\cite{Wu19A}. After their breakthrough, a new proof \cite{Austin22} and  numerous extensions and related works have appeared; we mention only a few~\cite{Algom20,AlgomWu23,Yu21A}.

If we restrict to missing-digit sets, Shmerkin and Wu's result can be stated as follows: let $b_1,b_2\geq 2$ be integers with $\log b_1/\log b_2\notin \Q$, and let $D_1\subset \{0,1,\ldots, b_1-1\}$ and $D_2\subset \{0,1,\ldots, b_2-1\}$ with $2\leq \#D_1<b_1$ and $2\leq\#D_2<b_2$,  then for any $\gamma,\alpha\in \R$, 
\[\dim_{\rm H}\left((\gamma K_{b_1,D_1}+\alpha)\cap K_{b_2,D_2}\right)\leq \max\{\dim_{\rm H}K_{b_1,D_1}+\dim_{\rm H}K_{b_2,D_2}-1,0\}.\]
Notably, prior to this result, Feng, Huang, and Rao \cite{FWR14A} proved the following weaker but still non-trivial inequality: under the same assumption, we have
\begin{equation}\label{eqdimFHR}
\dim_{\rm H}\left((\gamma K_{b_1,D_1}+\alpha)\cap K_{b_2,D_2}\right)< \min\{\dim_{\rm H}K_{b_1,D_1},\dim_{\rm H}K_{b_2,D_2}\}
\end{equation}
for any $\gamma,\alpha\in \R$; see  \cite[Theorem 1.6]{FWR14A}. It is helpful to compare \eqref{eqdimFHR} with our result Theorem~\ref{thmdimensiondrop}. We can understand the strict inequality  \eqref{eqdimFHR} as a dimension drop property for intersections of missing-digit sets with  incommensurable defining bases (i.e., $\log b_1/\log b_2\notin \Q$). 	Certainly, \eqref{eqdimFHR} does not apply when the two missing-digit sets are identical. However, our result Theorem~\ref{thmdimensiondrop} establishes a quantitative dimension drop property for the intersection of a single missing-digit set $K_{b,D}$ with its affine copies, which arises from an arithmetic non-resonance between the rational slope $\gamma$ and the base  $b$.

There have been many other related works on the slices of fractal sets. Many of them investigate the dimension formula valid for almost all  slices with a fixed direction, so that one can  compare with the value predicted by Marstrand's Slicing Theorem (often referred to as Marstrand's value); see, for instance, \cite{Hawkes75S,KenyonPeres91I,WWX13D,WuXi13D}. Slices of classical fractals such as the Sierpi{\'n}ski gasket and the Sierpi{\'n}ski carpet have also been studied; see \cite{BFS12S,LiuXiZhao07D,MS13D}.  Our result Theorem~\ref{thmdimensiondrop} differs from these studies in that we prove that, under a natural arithmetic condition on the slope, a quantitative dimension drop occurs for every slice in the prescribed direction, rather than merely for almost every slice.

We end the introduction with a brief outline of the proof of
Theorem~\ref{thmdimensiondrop}. The proof has two main ingredients. First, we reduce the number of basic intervals at a given  level that  meet $(\gamma K_{b,D}+\alpha)\cap K_{b,D}$ to the number of solutions to certain Diophantine equations. These solution counts are then controlled using congruence reductions and a block-counting argument; see Lemma~\ref{lem:block-counting} and the beginning of the proof of Theorem~\ref{thmdimensiondrop} in Section~\ref{S3}. The other key ingredient is to show that the quantity $M_{r,s}(q)$ in Lemma~\ref{lem:block-counting} is strictly smaller than the cardinality of $D_q$. This strict inequality is what yields the dimension drop; see Lemma~\ref{lem:Mqrsinq}.

The paper is organized as follows. In Section~\ref{S2}, we study the dimension drop for $C^1$-intersections of self-similar sets on $\R$. We prove Theorem~\ref{thmC1drop} and its consequence Corollary~\ref{CorEGamma} in this section. In the rest of the paper, we study the dimension drop for affine intersections of missing-digit sets. We first establish  several preliminary lemmas in Section~\ref{S3}. Then we present the proofs of Theorem~\ref{thmdimensiondrop} and Corollary~\ref{CorChar} in Section~\ref{S4}.

\section{Dimension drop for \texorpdfstring{$C^1$}{C1}-intersections: proof of Theorem~\ref{thmC1drop}}\label{S2}

In this section, we study the dimension drop for $C^1$-intersections of homogeneous self-similar sets on $\R$ satisfying the OSC. Our target is to prove Theorem~\ref{thmC1drop}. 

Throughout this section, let \(E\subset\R\) be the self-similar set  generated by a homogeneous IFS
$\Phi=\{\phi_i(x)=\rho x+a_i\}_{i=1}^m$
satisfying the OSC, where $0<\rho<1$ and $a_i\in\R$, $i=1,\ldots, m$.
Let $\Sigma=\{1,\ldots,m\}$. 
For $n\in\N$ and a word $
I=i_1\ldots i_n\in\Sigma^n$,
write
\[
|I|=n,
\qquad
\phi_I=\phi_{i_1}\circ\cdots\circ\phi_{i_n},
\qquad
E_I=\phi_I(E).
\]
Then
\[
\phi_I(x)=\rho^{|I|}x+a_I
\]
for some translation \(a_I\in\R\). Write $\Sigma^*=\bigcup_{n\in \N}\Sigma^n$.

A key step to prove Theorem~\ref{thmC1drop}
is to establish the following lemma, which is a same-set and slope-tracking version of \cite[Proposition 2.2]{FWR14A}.

For \(n\ge1\), define
\begin{equation}\label{eqdefsn}
	s_n=\frac{\log(m^n-1)}{n\log(1/\rho)}.
\end{equation}
Note that
\[
s_n<s:=\dim_{\rm H}E=\frac{\log m}{\log(1/\rho)}
\quad \text{ and } \quad 
\lim_{n\to\infty}s_n=s.
\]

\begin{lem}\label{lemBlowup}
 Let \(f:\R\to\R\) be a \(C^1\)-diffeomorphism and put \(h=f^{-1}\). There exist integers $M_0, p_0\in\N$ such that the following statement holds. Fix \(n\in\N\). If
	\[\overline{\dim}_{\rm M}(E\cap h(E))>s_n,\]
	then there exist an integer \(k\le M_0\), affine maps $g_1,\ldots,g_k:\R\to\R$, points $ z_1,\ldots,z_k\in f(E)\cap E,$
	and an integer \(p\in\Z\) with $|p|\leq p_0$, such that the following hold.
	\begin{enumerate}
		\item[{\rm (i)}] For every word \(J\in\Sigma^n\),
		\[
		E_J\cap\left(\bigcup_{i=1}^k g_i(E)\right)\neq\emptyset.
		\]
		\item[{\rm (ii)}] For $1\leq i\leq k$, the linear part of  $g_i$ is $g_i'=\rho^p h'(z_i)$.
	\end{enumerate}
	Here $M_0$ and $p_0$ are independent of $n$.
\end{lem}

To prove Lemma~\ref{lemBlowup}, we follow the localization and recalling argument in the proof of \cite[Proposition 2.1]{FWR14A}, and in the meantime,  we track the slopes of the affine maps generated in the recalling process. We need the following lemma proved in \cite{FWR14A}, which provides the integer $M_0$ in Lemma~\ref{lemBlowup}. 

For \(0<r<\rho\), define 
\[
\mathcal A_r=\{I\in\Sigma^n: n\in\N, \rho^{n}\le r<\rho^{n-1}\}.
\]

\begin{lem}\cite[Lemma~2.1]{FWR14A}\label{lemOSC}
	There exists \(M_0\in\N\) such that for every \(0<r<\rho\) and every \(I\in\mathcal A_r\),
\[
	\#\{J\in\mathcal A_r:{\rm dist}(E_I,E_J)\le r\}\le M_0,
	\]
    where ${\rm dist}(E_I,E_J) := \inf\{|x-y| : x \in E_I, y \in E_J\}$.
\end{lem}

\begin{proof}[Proof of Lemma~\ref{lemBlowup}]
	We claim that for every \(j\in\N\) there exists a word \(W_j\in\Sigma^*\) with $
	|W_j|\ge j$
	such that
	\begin{equation}\label{eqEWjJ}
		E_{W_jJ}\cap h(E)\neq\emptyset
		\qquad\text{for every }J\in\Sigma^n.
	\end{equation}
	Suppose this fails. Then there exists \(j_0\in\N\) such that for every word \(W\) with \(|W|\ge j_0\),  there is at least one  \(J\in\Sigma^n\) such that \[
	E_{WJ}\cap h(E)=\emptyset.\]
	For \(q\ge1\), let \(\Gamma_q\) be the collection of words of the form $
	UJ_1\ldots J_q$, 
	where \(|U|=j_0\), each \(J_i\in\Sigma^n\), and
	\[
	E_{UJ_1\ldots J_q}\cap h(E)\neq\emptyset.
	\]
	Note that the collection \(\{E_I:I\in\Gamma_q\}\) covers \(E\cap h(E)\). Moreover,   $\#\Gamma_1\leq m^{j_0}(m^n-1)$ and 
	\[\#\Gamma_{k+1}\leq (m^n-1)\#\Gamma_k,\qquad \forall k\geq 1.\]
	Hence
	\[\#\Gamma_q\leq m^{j_0} (m^n-1)^q,\qquad  \forall  q\geq1. \]
	Therefore, 
	\[\overline{\dim}_{\rm M}(E\cap h(E))\leq \limsup_{q\to\infty}\frac{\log(m^{j_0}(m^n-1)^q)}{\log(1/\rho^{j_0+nq})}=s_n,\]
	contradicting the hypothesis. This proves \eqref{eqEWjJ}.
	
Since \(h\) is a \(C^1\)-diffeomorphism, we have $
	m_h:=\min_{x\in {\rm conv}(E)} |h'(x)|>0$.
	Define
	\[
	r_j=\frac{\rho^{|W_j|}{\rm diam}E}{m_h},
	\]
where ${\rm diam}E$ is the diameter of $E$. 
	By omitting the first few $j$'s, we can assume $0<r_j<\rho$ for $j\geq1$. 
	Let \(N_j\in\N\) be the unique integer such that
	\begin{equation}\label{eqNj}
		\rho^{N_j}\le r_j<\rho^{N_j-1}.
	\end{equation}
	Let $$
	\mathcal I_j=\{I\in\Sigma^{N_j}:E_I\cap h^{-1}(E_{W_j})\neq\emptyset\}.$$
	We claim that there exists \(M_0\in\N\) such that 
\begin{equation}\label{eqIjlessN0}
		\#\mathcal I_j\le M_0.
	\end{equation}
	Indeed, if \(I,I'\in\mathcal I_j\), choose \(z\in E_I\), \(z'\in E_{I'}\) with $h(z),h(z')\in E_{W_j}$.
	Then
	\[
	|h(z)-h(z')|\leq\rho^{|W_j|}{\rm diam}E.
	\]
	Since \(|h(u)-h(v)|\ge m_h|u-v|\) on \({\rm conv}(E)\), we get
	$|z-z'|\leq r_j.$
	Thus \({\rm dist}(E_I,E_{I'})\leq r_j\). Then \eqref{eqIjlessN0} follows from Lemma~\ref{lemOSC}.
	
	Passing to a subsequence in \(j\), we may assume that for every $j\geq 1$, $
	\#\mathcal I_j=k$
	for a fixed \(k\le M_0\). Write $\mathcal I_j=\{I_{j,1},\ldots,I_{j,k}\}$.
	For \(1\leq i\leq k\), define
	\[
	h_{j,i}=\phi_{W_j}^{-1}\circ h\circ\phi_{I_{j,i}}.
	\]
	We show that for every \(J\in\Sigma^n\),
	\begin{equation}\label{eqEJint}
		E_J\cap\left(\bigcup_{i=1}^k h_{j,i}(E)\right)\neq\emptyset.
	\end{equation}
	To see this, fix $J\in \Sigma^n$.
	By \eqref{eqEWjJ}, we can choose $
	y\in E_{W_jJ}\cap h(E)$.
	Write \(y=h(x)\) with \(x\in E\). Since \(y\in E_{W_j}\), we have $
	x\in E\cap h^{-1}(E_{W_j})$.
	Thus \(x\in E_{I_{j,i}}\) for some \(i\), say \(x=\phi_{I_{j,i}}(u)\) with \(u\in E\). Since \(y\in E_{W_jJ}\), $
	\phi_{W_j}^{-1}(y)\in E_J.$
	Also note that 
	\[
	\phi_{W_j}^{-1}(y)=\phi_{W_j}^{-1}h\phi_{I_{j,i}}(u)=h_{j,i}(u)\in h_{j,i}(E).
	\]
	This proves \eqref{eqEJint}.

	Fix \(x_0\in E\). Define an affine map
	\[
	A_{j,i}(x)=h_{j,i}(x_0)+h_{j,i}'(x_0)(x-x_0).
	\]
	Since $
	\phi_{I_{j,i}}'(x)=\rho^{N_j}$
	and $\phi_{W_j}'(x)=\rho^{|W_j|},$
	we have 
	\[
	h_{j,i}'(x)=\rho^{N_j-|W_j|}h'(\phi_{I_{j,i}}(x)),
	\]
	and so
	\begin{equation}\label{eqAjislope}
		A_{j,i}'=\rho^{N_j-|W_j|}h'(\phi_{I_{j,i}}(x_0)).
	\end{equation}
	Since $h$ is $C^1$ and  ${\rm diam}\phi_{I_{j,i}}({\rm conv}(E))=\rho^{N_j}{\rm diam}E\to0$ as $j\to \infty$, we see that, uniformly for \(x\in E\),
	\begin{align}
		|h_{j,i}(x)-A_{j,i}(x)|&=|h_{j,i}(x)-A_{j,i}(x)-(h_{j,i}(x_0)-A_{j,i}(x_0))|\nonumber\\
		&\leq \sup_{z\in {\rm conv}(E)}\left|\rho^{N_j-|W_j|}h'(\phi_{I_{j,i}}(z))-\rho^{N_j-|W_j|}h'(\phi_{I_{j,i}}(x_0))\right|\cdot{\rm diam }E\nonumber\\
		&\to 0 \qquad \text{ as } j\to\infty.\label{eqhjiAji0}
	\end{align}
	Here we have used that \(\rho^{N_j-|W_j|}\) is uniformly bounded, which follows from \eqref{eqNj}.

	Since \(h'\) is bounded and \(\rho^{N_j-|W_j|}\) is uniformly bounded,
	the linear parts \(A_{j,i}'\) are uniformly bounded.  The translations of $A_{j,i}$ are also uniformly bounded: by the definition of \(\mathcal I_j\), for each \(i\in \{1,\ldots,k\}\) there exists $z_{j,i}\in E_{I_{j,i}}$ such that $
	h(z_{j,i})\in E_{W_j}$.
	Writing \(z_{j,i}=\phi_{I_{j,i}}(u_{j,i})\) with \(u_{j,i}\in E\), we obtain
	\[
	h_{j,i}(u_{j,i})=h_{j,i}\phi_{I_{j,i}}^{-1}(z_{j,i})=\phi_{W_j}^{-1}h(z_{j,i})\in E.
	\]
	Hence \(h_{j,i}(E)\cap E\neq\emptyset\). Therefore, the translation parts of \(A_{j,i}\) are uniform bounded.
	
	Passing to a further subsequence, for each \(i\in \{1,\ldots,k\}\) there is an affine map $g_i$ such that 
	$A_{j,i}\to g_i$
	uniformly on compact sets. By \eqref{eqhjiAji0}, the compact sets \(h_{j,i}(E)\) converge in Hausdorff distance to \(g_i(E)\). Taking limits in \eqref{eqEJint}, we get
	\[
	E_J\cap\left(\bigcup_{i=1}^k g_i(E)\right)\neq\emptyset
	\]
	for every \(J\in\Sigma^n\).
	
	Finally, we extract the slopes of $g_i$. 
	Since \(\rho^{N_j-|W_j|}\) is uniformly bounded and  \(N_j-|W_j|\in\Z\), it can take only finitely many values. By passing to a further subsequence, we can assume for every $j\geq 1$, 
	$N_j-|W_j|=p$
	for some fixed \(p\in\Z\). From \eqref{eqAjislope},
	\begin{equation}\label{eqgiprime}
		g_i'=\lim_{j\to\infty}A_{j,i}'
		=\rho^p\lim_{j\to\infty}h'(\phi_{I_{j,i}}(x_0)).
	\end{equation}
	Next we replace \(\phi_{I_{j,i}}(x_0)\) by a point in $E\cap f(E)$. As above, choose $
	z_{j,i}\in E_{I_{j,i}}$
	with
	$h(z_{j,i})\in E_{W_j}.$
	Then $z_{j,i}\in E\cap h^{-1}(E)=E\cap f(E)$.
	Moreover,
	\[
	|z_{j,i}-\phi_{I_{j,i}}(x_0)|\le \rho^{N_j}{\rm diam}E\to0.
	\]
	Passing to a further subsequence, we assume $z_{j,i}\to z_i\in E\cap f(E)$.
	Then by the continuity of \(h'\), \eqref{eqgiprime} becomes $
	g_i'=\rho^p h'(z_i)$.
	This proves the lemma.
\end{proof}

The following consequence of Lemma~\ref{lemBlowup} enables us to apply a result of Feng and Wang~\cite{FengWang09O} on affine embeddings of self-similar sets, which is an  important  step to derive Theorem~\ref{thmC1drop}. 

\begin{lem}\label{lemaffineembedding}
	Assume
	\[
	\overline{\dim}_{\rm M}(f(E)\cap E)=\dim_{\rm H} E.
	\]
	Then there exist an affine map \(\psi(x)=\lambda x+t\), a point $x_0\in E\cap f^{-1}(E)$,
	and an integer $q\in\Z$, such that $\psi(E)\subset E$ and  $  
	|\lambda|=\rho^q|f'(x_0)|$.
\end{lem}

\begin{proof}
	Put $h=f^{-1}$.	Since $f$ is a $C^1$ diffeomorphism on $\R$, we have 
	\[
	\overline{\dim}_{\rm M}(E\cap h(E))=\dim_{\rm H} E:=s.
	\]
	For $n\in\N$, let $s_n$ be defined in \eqref{eqdefsn}. 
	Since \(s_n<s\),  Lemma~\ref{lemBlowup} applies. Thus for each  \(n\in\N\) there exist affine maps $g_{n,1},\ldots,g_{n,k_n}$ with $k_n\le M_0$ and $g_{n,i}'=\rho^{p_{n}}h'(z_{n,i})$ for some integer $p_{n}$ and some point $z_{n,i}\in E\cap f(E)$, such that
	\begin{equation}\label{eqEjcapi1k}
		E_J\cap \left(\bigcup_{i=1}^{k_n}g_{n,i}(E)\right)\neq \emptyset.
	\end{equation}
	Since $E\cap f(E)$ is compact, $k_n\leq M_0$, $|p_{n}|\leq p_0$ and $M_0,p_0$ are independent of $n$, 
	by passing to a subsequence in \(n\), we may assume $k_n=k$,
	all \(p_{n}=p\) are fixed in \(n\),  $
	g_{n,i}\to g_i$, and $
	z_{n,i}\to z_i\in E\cap f(E)$ for each \(1\le i\le k\).
	Then $
	g_i'=\rho^{p}h'(z_i)$.
	We claim that 
	\begin{equation}\label{EsubcupikgiE}
		E\subset\bigcup_{i=1}^k g_i(E).
	\end{equation}
	Let \(x\in E\). For each selected \(n\), let \(J_n\in\Sigma^n\) be a word with \(x\in E_{J_n}\). By \eqref{eqEjcapi1k}, choose $
	x_n\in E_{J_n}\cap g_{n,i_n}(E)$
	for some \(i_n\). Since \(k\) is finite, along a subsequence depending on \(x\), \(i_n=i\) is constant. Because \({\rm diam} E_{J_n}\to0\), we have \(x_n\to x\). Since \(g_{n,i}(E)\to g_i(E)\) in the Hausdorff metric, \(x\in g_i(E)\). This proves \eqref{EsubcupikgiE}.
	
	Now apply the Baire category theorem to the compact metric space \(E\). Since \(E\) is covered by finitely many closed sets \(E\cap g_i(E)\), one of them has nonempty relative interior in \(E\). Hence there exist  \(i\in\{1,\ldots,k\}\) and an open set \(V\subset\R\) such that $
	\emptyset\neq E\cap V\subset g_i(E)$.
Therefore, there exists a word \(J\in \Sigma^*\)  such that
	\[\phi_J(E)\subset E\cap V\subset g_i(E).
	\]
	Define $
	\psi=g_i^{-1}\circ\phi_J$.
	Then $\psi$ is an affine map satisfying $
	\psi(E)\subset E$ and 
	\[
	|\psi'|=\frac{|\phi_J'|}{|g_i'|}
	=\frac{\rho^{|J|}}{\rho^{p}|h'(z_i)|}
	=\frac{\rho^{|J|-p}}{|h'(z_i)|}.
	\]
	Letting \(q=|J|-p\) and $x_0\in E\cap f^{-1}(E)$ such that $f(x_0)=z_i$, we complete the proof.
\end{proof}

With Lemma~\ref{lemaffineembedding} in hand, we can apply the following Logarithmic Commensurability Theorem due to Feng and Wang~\cite{FengWang09O} to derive Theorem~\ref{thmC1drop}.

\begin{thm}\cite[Theorem 1.1]{FengWang09O}\label{thmFW}
	Let \(E\subset\R\) be the self-similar set  generated by a homogeneous IFS $\Phi$ with contraction ratio $\rho\in (0,1)$. Suppose that \(\Phi\) satisfies the OSC and $\dim_{\rm H} E<1$.
	If \(\psi(x)=\lambda x+t\), \(\lambda\ne0\), satisfies $\psi(E)\subset E$,
	then $
	\log|\lambda|/\log\rho\in\Q$.
\end{thm}

\begin{proof}[Proof of Theorem~\ref{thmC1drop}]
	Assume on the contrary that $\overline{\dim}_{\rm M}\left(f(E)\cap E\right)=\dim_{\rm H}E$.
	Then by Lemma~\ref{lemaffineembedding}, there exists an affine map \(\psi(x)=\lambda x+d\) such that $\psi(E)\subset E$ and  $  
	|\lambda|=\rho^q|f'(x_0)|$ for some $q\in\Z$ and  \(x_0\in E\cap f^{-1}(E)\).  It then follows from  Theorem~\ref{thmFW} that $
	\log |f'(x_0)|/\log\rho\in\Q$, contradicting the hypothesis. This completes the proof of the theorem. 
\end{proof}

To end this section, we prove Corollary~\ref{CorEGamma}.

\begin{proof}[Proof of Corollary~\ref{CorEGamma}]
Since $\rho^{-1}$ is a prime number, it is a fact from elementary number theory that if $\gamma\in\Q$, $\gamma>0$ and $\gamma\not\in \{\rho^{k}: k\in\Z\}$, then $\log\gamma/\log\rho\not\in \Q$.  Hence the ``if'' part follows from Theorem~\ref{thmC1drop}. To see the ``only if'' part, write $\rho^{-1}=p$. Let $\gamma=p^k$ for some integer $k\geq0$ (the case that $k<0$ is proved similarly). Rewrite the identity $E=\bigcup_{i=1}^m(\rho E+a_i)$
as
\begin{equation}\label{eqPEEA}
pE=E+A,
\end{equation} 
where $A=\{\rho^{-1}a_i: 1\leq i\leq m\}$ and for $U,V\subset \R$, $U+V:=\{u+v: u\in U, v\in V\}$. Iterating \eqref{eqPEEA} gives 
\[p^jE=E+A_j,\quad j\geq 0,\]
where $A_0:=\{0\}$ and $A_j:=A+pA+\cdots+p^{j-1}A$ for $j\geq1$. Pick $\alpha\in -A_k$. Then we have
\[(\gamma E+\alpha)\cap E=(E+A_k+\alpha)\cap E\supset E.\]
Hence \eqref{eqgEaEless} does not hold. This proves the corollary.  
\end{proof} 

\section{Dimension drop for affine intersections of missing-digit sets: Preliminary lemmas}\label{S3}

Following the strategy described at the end of the Introduction, in this section, we establish several lemmas that are needed to prove Theorem~\ref{thmdimensiondrop}. The main results in this section are Lemmas~\ref{lem:block-counting} and ~\ref{lem:Mqrsinq}.

We first give some notation. 
For $m\geq 0$, put
\[
D_m=\left\{\sum_{j=0}^{m-1}\varepsilon_jb^j:\varepsilon_j\in D, 0\leq j\leq m-1\right\},
\]
with the convention that $D_0=\{0\}$.
Equivalently, $D_m$ is the set of integer codings of the left endpoints of the level $m$
basic intervals in the construction of $K_{b,D}$: those endpoints are $A/b^m$ with $A\in D_m$, and we have
\[K_{b,D}=\bigcap_{m=1}^{\infty}\bigcup_{A\in D_m}\left[\frac{A}{b^m}, \frac{A+1}{b^m}\right].\]
Notice that
$\#D_m=(\#D)^m$.

For subsets $E,F\subset \mathbb Z$ and an integer $N\geq 1$, we write
$E\equiv F \pmod{N}$
if their images in $\mathbb Z/N\mathbb Z$ are equal. Similarly, we write
$E\subseteq F\pmod{N}$
if the image of $E$ in $\mathbb Z/N\mathbb Z$ is contained in the image of $F$.

For  $q\in \N$ and $r,s\in\Z$,   define
\begin{equation}\label{eqMqrs}
	M_{r,s}(q):=\max_{c\in \Z}
	\#\{(u,v)\in D_q^2:ru-sv\equiv c\pmod{b^q}\}.
\end{equation}
In the following lemma, we apply $M_{r,s}(q)$ to bound the number of solutions $(A,B)\in D_m^2$ to the Diophantine equation $rA-sB=h$ for every fixed $h\in \Z$, which arises naturally in counting basic intervals of level $m$ that intersect $(\gamma K_{b,D}+\alpha)\cap K_{b,D}$; see the beginning of the proof of Theorem~\ref{thmdimensiondrop} in Section~\ref{S3}. 

\begin{lem}\label{lem:block-counting}
Let  $q\in \N$ and $r,s\in \Z$. Let $M_{r,s}(q)$ be defined as in \eqref{eqMqrs}.  For $m\in \N$ write
$m=k q+j$, where $k\geq 0$ and $0\leq j <q$.
	Then for every $h\in\Z$,
	\[
	\#\{(A,B)\in D_m^2:rA-sB=h\}\leq (\#D)^{2q} M_{r,s}(q)^{m/q}.
	\]
\end{lem}

\begin{proof}
Set $Q=b^q$.	For $A,B\in D_m$, we rewrite  their $b$-ary expansions in blocks of length $q$:
	\[
	A=a_0+a_1Q+\cdots+a_{k-1}Q^{k-1}+A_*Q^k,
	\]
	\[
	B=b_0+b_1Q+\cdots+b_{k-1}Q^{k-1}+B_*Q^k,
	\]
	where $a_i,b_i\in D_q$ and $A_*,B_*\in D_{j}$. Suppose that
	\begin{equation}\label{eqAkBh}
		rA-sB=h.
	\end{equation}
	Set $h_0=h$. Reducing the equation modulo $Q$ gives
	\[
	ra_0-sb_0\equiv h_0\pmod Q.
	\]
By the definition of $M_{r,s}(q)$,  there are at most $M_{r,s}(q)$ choices for $(a_0,b_0)$. Once $(a_0,b_0)$ is chosen, define
	\[
	h_1=\frac{h_0-ra_0+sb_0}{Q}\in\Z.
	\]
	Then \eqref{eqAkBh} implies that 
	\[\frac{r(A-a_0)}{Q}-\frac{s(B-b_0)}{Q}=h_1,\]
which, modulo $Q$, yields that
\[ra_1-sb_1\equiv h_1\pmod{Q}.\]
Continuing the above argument, we see that for  each $0\leq i\leq k-1$, there exists $h_i\in\Z$ such that 
	\[
	ra_i-sb_i\equiv h_i\pmod Q
	\]
	and so there are at most $M_{r,s}(q)$ choices of $(a_i,b_i)$. Hence the $k$ full blocks contribute at most
	$M_{r,s}(q)^k$ choices. The remaining blocks $(A_*,B_*)\in D_{j}^2$ contribute at most
	$(\#D_{j})^2=(\#D)^{2j}$ choices. Therefore,
	\[
	\#\{(A,B)\in D_m^2:rA-sB=h\}\leq (\#D)^{2j} M_{r,s}(q)^k\leq (\#D)^{2q} M_{r,s}(q)^{m/q},
	\]
since $j<q$, $k\leq m/q$ and $M_{r,s}(q)\geq1$.
\end{proof}

Another key ingredient in the proof of Theorem~\ref{thmdimensiondrop} is to show that  for all sufficiently large $q$, $M_{r,s}(q)$ is strictly less than $\#D_q$, which  leads to the dimension drop in Theorem~\ref{thmdimensiondrop}. We achieve this by establishing the following.  

\begin{lem}\label{lem:Mqrsinq}
	Let $r,s\in\Z$ satisfy $|r|\neq |s|$ and $\gcd(rs,b)=1$. Let $M_{r,s}(q)$ be defined as in \eqref{eqMqrs}. 
	Then for every $q \in \N$ satisfying $
	(|r|+|s|)(\#D)^{q-1}<b^{q-1}$,
	we have
	\[M_{r,s}(q)\leq (\#D)^q-1<(\#D)^q.\]
\end{lem}

To prove Lemma~\ref{lem:Mqrsinq}, we  need  several preliminary  results. The following lemma is a simple fact on translation-invariant subsets of $\Z/N\Z$, $N\in \N$.

\begin{lem}\label{lem:bounded_translation}
	Let $T,N\in \N$. Let $F\subset \Z/N\Z$ be a non-empty set with $T(\#F)<N$.
Then there is no nonzero integer $t$ with $|t|\leq T$ such that $F+t\equiv F \pmod{N}$.
\end{lem}

\begin{proof}
Suppose on the contrary that there exists a nonzero integer $t$ with
	$|t|\leq T$ such that $F+t\equiv F\pmod{N}$.
	Since $T(\#F)<N$ and $F$ is non-empty, we clearly have 
	$0<|t|<N$ and  so $t\not\equiv 0\pmod{N}$.
	
	Consider the translation map $T_t:\Z/N\Z\to \Z/N\Z$,
	$x\mapsto x+t \pmod{N}$.
	Every orbit of $T_t$ has length $N/\gcd(t,N)$.
	Since $F+t\equiv F\pmod{N}$, $F$ contains the orbit of every element in $F$.  Thus we have $
	\#F\geq N/\gcd(t,N)$.
Since $
	\gcd(t,N)\leq |t|\leq T$, this implies that $\#F\geq N/T$,
	which contradicts $T(\#F)<N$. This proves the lemma.
\end{proof}

The key step to prove Lemma~\ref{lem:Mqrsinq} is the following lemma. 

\begin{lem}\label{lemrDqsDqc}
Let $r,s\in\Z$ satisfy $|r|\neq |s|$ and $\gcd(rs,b)=1$. 
Let $q\in \N$ be such that $(|r|+|s|)(\#D)^{q-1}<b^{q-1}$.
Then $sD_q+c\not\equiv rD_q\pmod{b^q}$
	for every $c\in\Z$.
\end{lem}

\begin{proof}
	We argue by contradiction. Suppose that for some $c\in\Z$, one has
\begin{equation}\label{rsmodD}
sD_q+c\equiv rD_q\pmod {b^q}. 
\end{equation}
	Write $c=c_0+bc_1$,
	where $0\leq c_0<b$ and 
	 $c_1\in\Z$.
	Reducing \eqref{rsmodD} modulo $b$, we get
\begin{equation}\label{eqrDc0sDb}
sD+c_0\equiv rD\pmod b. 
\end{equation}
Since $\gcd(rs,b)=1$,  $r$ and $s$ are both coprime to $b$. Hence the sets $sD+c_0$ and $rD$ both have
	exactly $\#D$ distinct residue classes modulo $b$. Thus \eqref{eqrDc0sDb} induces a
	bijection $
	\pi:D\to D$
	such that
	\[
	sa+c_0\equiv r\pi(a)\pmod b, \qquad \forall a\in D.
	\]
For each $a\in D$, define 
	\[
	\tau(a):=\frac{sa+c_0-r\pi(a)}{b}\in\Z.
	\]
	
Note that $D_q$ admits the following disjoint decomposition
\begin{equation}\label{eqDecomDq}
	D_q=\bigsqcup_{a\in D}\left(a+bD_{q-1}\right).
\end{equation}
For $a\in D$ and $A'\in D_{q-1}$, we have
	\begin{equation*}
s(a+bA')+c=sa+c_0+b(sA'+c_1)=r\pi(a)+b\left(sA'+c_1+\tau(a)\right).
	\end{equation*}
	By \eqref{rsmodD}, this element is congruent modulo $b^q$ to an element of $rD_q$.
	Moreover, since its residue modulo $b$ is $r\pi(a)$, we see from \eqref{eqDecomDq} that $sA'+c_1+\tau(a)$ must
	belong to $rD_{q-1}$ modulo $b^{q-1}$. Hence,
\begin{equation}\label{eqrDqm1taua}
sD_{q-1}+c_1+\tau(a)
\subseteq
rD_{q-1}
\pmod{b^{q-1}}. 
\end{equation}
	Since $r$ and $s$ are both coprime to $b$, both sides of \eqref{eqrDqm1taua} have exactly $(\#D)^{q-1}$
	residue classes modulo $b^{q-1}$. 
So we have equality 
\begin{equation}\label{eqrDqm1c1a}
sD_{q-1}+c_1+\tau(a)
\equiv
rD_{q-1}
\pmod{b^{q-1}}.
\end{equation}
This holds for every $a\in D$.

Applying \eqref{eqrDqm1c1a} to  two digits $a,a'\in D$, we see that 
	\[
	sD_{q-1}+c_1+\tau(a)
	\equiv
	sD_{q-1}+c_1+\tau(a')
	\pmod{b^{q-1}}.
	\]
	Equivalently, let $F$ be the image of $sD_{q-1}+c_1+\tau(a')$ in
	$\Z/b^{q-1}\Z$. Then
	\begin{equation}\label{eqEq1E}
F+\tau(a)-\tau(a')\equiv F\pmod{b^{q-1}}.
	\end{equation}
Moreover, again since $s$ is coprime to $b$, we have $
	\#F=\#D_{q-1}=(\#D)^{q-1}$. On the other hand, note that
	\[
	\tau(a)-\tau(a')
	=
	\frac{s(a-a')-r(\pi(a)-\pi(a'))}{b}.
	\]
	Since $
	a,a',\pi(a),\pi(a')\in D\subset\{0,1,\ldots,b-1\}$,
	we have
	\begin{equation}\label{eqtauaad}
		|\tau(a)-\tau(a')|\leq \frac{|s|(b-1)+|r|(b-1)}{b}<|s|+|r|.
	\end{equation}
In view of \eqref{eqEq1E}, \eqref{eqtauaad} and the assumption that $	(|r|+|s|)(\#D)^{q-1}<b^{q-1}$, it follows from 
 Lemma~\ref{lem:bounded_translation} (in which we take $F$ as above and $T=|r|+|s|$) that $\tau(a)-\tau(a')=0$ for any $a,a'\in D$.
		
Therefore, $\tau(\cdot)$ is constant on $D$. Write the common value as $\tau_0$.
	Then for every $a\in D$, we have $
	sa+c_0-r\pi(a)=b\tau_0$,
	or equivalently, $sa+c_0-b\tau_0=r\pi(a)$.
	Since $\pi:D\to D$ is a bijection, we obtain an equality of subsets of $\Z$:
	\[
	sD+c_0-b\tau_0=rD. 
	\]
	Taking the diameters on both sides,  we get
	$|s|\,{\rm diam}D=|r|\,{\rm diam}D$.
	Since ${\rm diam}D>0$ as $\#D\geq 2$, this forces $|s|=|r|$, contradicting the assumption that $|s|\neq |r|$.
	This proves the lemma.
\end{proof}

With Lemma~\ref{lemrDqsDqc} in hand,  we are now ready to give the proof of Lemma~\ref{lem:Mqrsinq}.  

\begin{proof}[Proof of Lemma~\ref{lem:Mqrsinq}]
Fix $c\in\Z$. For each $v\in D_q$, the congruence equation $ru-sv\equiv c\pmod{b^q}$ has a unique solution $u\in\Z/b^q\Z$, because $\gcd(r,b)=1$. Hence it has at	most one solution $u\in D_q$. Therefore,
\begin{equation}\label{eqUVdQ2Surv}
\#\left\{(u,v)\in D_q^2:ru-sv\equiv c\pmod{b^q}\right\}\leq\#D_q=(\#D)^q.
\end{equation}
This proves that $M_{r,s}(q)\leq (\#D)^q$ for every $q\in\N$.

To prove the lemma, we argue by contradiction. 
Suppose the conclusion fails. Then there exists $q\in\N$ satisfying $
(|r|+|s|)(\#D)^{q-1}<b^{q-1}$ such that  $M_{r,s}(q)=(\#D)^q.$ So equality holds in \eqref{eqUVdQ2Surv} for some $c\in \Z$. Thus for every $v\in D_q$ there exists
	$u\in D_q$ such that
	$ru-sv\equiv c\pmod{b^q}$,
	which implies that
$sD_q+c\subseteq rD_q\pmod{b^q}$.
	Since $\gcd(rs,b)=1$, $r$ and $s$ are both coprime to $b$. It follows that  $sD_q+c$ and $rD_q$ both have exactly $\#D_q=(\#D)^q$
	residue classes modulo $b^q$. Thus we have equality $sD_q+c\equiv rD_q\pmod{b^q}$. However, this contradicts Lemma~\ref{lemrDqsDqc}. We complete the proof. 
\end{proof}

\section{Proofs of Theorem \ref{thmdimensiondrop} and Corollary~\ref{CorChar}}\label{S4}

We first prove Theorem~\ref{thmdimensiondrop} using Lemmas~\ref{lem:block-counting} and~\ref{lem:Mqrsinq}.

\begin{proof}[Proof of Theorem~\ref{thmdimensiondrop}]
We first treat the case $\ell=0$. Then $\gamma=\frac{s}{r}$, where $r\in\N$, $s\in\Z $ such that  $|s|\neq r$ and $\gcd(s,r)=\gcd(sr,b)=1$. The case $\ell\neq 0$ will be reduced to this case at the end of the proof. 

To ease notation, write $K=K_{b,D}$. 
For $m\in\N$ and $A \in D_{m}$, let $$I_{m,A}:= \left[\frac{A}{b^{m}}, \frac{A+1}{b^{m}}\right].$$ Then
	$$K\subset \bigcup_{A \in D_{m} } I_{m,A}.$$
Set $$E_{m}(\gamma,\alpha)=\{A \in D_{m} \colon I_{m,A} \cap (\gamma K+\alpha)\cap K\neq \emptyset\}.$$
It is clear that 
\begin{equation}\label{eqGKaKupp}
(\gamma K+\alpha)\cap K\subset \bigcup_{A\in E_m(\gamma, \alpha)}I_{m,A}.
\end{equation}
Below we estimate the cardinality of $E_m(\gamma,\alpha)$. 

Let $A \in E_{m}(\gamma,\alpha)$. Then there exists $x \in I_{m,A} \cap (\gamma K+\alpha) \cap K$. Since $\frac{x-\alpha}{\gamma}\in K$, we have 
	$$\frac{x-\alpha}{\gamma} \in I_{m,B} = \left[\frac{B}{b^{m}}, \frac{B+1}{b^{m}}\right] $$
	for some $B \in D_{m}$. Therefore,
	\[\left|x-\frac{A}{b^m}\right|\leq b^{-m}, \qquad \left|x-\left(\frac{\gamma B}{b^m}+\alpha\right)\right|\leq |\gamma| b^{-m}.\]
By the triangle inequality and using $\gamma=\frac{s}{r}$, we get
\begin{equation}\label{rAsBineq}
|rA-sB-\alpha rb^m|\leq |s|+r.
\end{equation}
That is,
	$$\alpha rb^m-(|s|+r)\leq rA-sB\leq \alpha rb^m+(|s|+r).$$
Let 
\[H=\{h\in \Z: \alpha rb^m-(|s|+r)\leq h\leq \alpha rb^m+(|s|+r)\}.\]
Let $q\in \N$ be such that 
\[q>\frac{\log(|s|+r)}{\log(b/(\#D))}+1.\]
Equivalently, $(|s|+r)(\#D)^{q-1}<b^{q-1}$.
	Then by Lemmas~\ref{lem:block-counting} and \ref{lem:Mqrsinq}, we have
	\begin{align}
		\# E_{m}(\gamma,\alpha) &\leq \#\{A\in D_m: \exists B\in D_m \text{ such that } \eqref{rAsBineq} \text{ holds}\}\nonumber\\&\leq  \sum_{h\in H}\#\{(A,B) \in D_{m}^{2} \colon rA-sB=h\}\nonumber\\
		&\leq (\#H) (\# D)^{2q} M_{r,s}(q)^{m/q} \nonumber\\
		&\leq (2r+2|s|+1)(\# D)^{2q} ((\# D)^{q}-1)^{m/q}.\label{eqBddEmag}
	\end{align}
Now it follows from \eqref{eqGKaKupp} and \eqref{eqBddEmag} that  
	$$\overline{\dim}_{\rm M} ((\gamma K+\alpha)\cap K) \leq \limsup_{m \to \infty} \frac{\log \# E_{m}(\gamma,\alpha)}{\log b^{m}} \leq \frac{\log ((\# D)^{q}-1) }{q\log b}.$$
This proves the theorem for the case that $\ell=0$. 

Now consider the case that  $\gamma=\frac{s}{r}b^{\ell}$ with $\ell\in\Z, \ell\neq 0$. Assume $r,s$ satisfy the same assumption as above. We further assume $\ell>0$; the other case that $\ell<0$ can be treated similarly.
Rewrite \eqref{eqKbDiden} as $bK=K+D$.
Iterating this gives $b^{\ell}K=K+D_{\ell}$.
Hence,
\begin{align*}
	(\gamma K+\alpha) \cap K&=\left(\frac{s}{r}b^{\ell}K+\alpha\right)\cap K\\
	&= \left(\frac{s}{r}K+\frac{s}{r}D_{\ell}+\alpha\right)\cap K\\
	&=\bigcup_{\alpha'\in \frac{s}{r}D_{\ell}+\alpha}\left(\frac{s}{r}K+\alpha'\right)\cap K.
\end{align*}
This is a finite union. 
Then the conclusion follows from the above result for  the case $\ell=0$. This completes the proof. 
\end{proof}

Finally, we give the proof of Corollary~\ref{CorChar}.

\begin{proof}[Proof of Corollary~\ref{CorChar}]
The ``only if'' part follows from the same arguments as in the proofs of Corollary~\ref{CorEGamma} and Remark~\ref{remps}(ii).
	To see the ``if'' part, let $\gamma\notin\{\pm b^{k}: k\in \Z\}$. The case $\gamma=0$ is trivial.  Assume $\gamma\neq 0$. Then it follows from Corollary~\ref{CorEGamma} and Remark~\ref{remps}(ii) when $\gamma\in\Q$, and  \cite[Corollary 6.4]{Shmerkin19} when $\gamma\not\in\Q$ that \eqref{eqstrinq} holds for every $\alpha\in\R$. This proves the corollary.
\end{proof}

{\noindent \bf  Acknowledgements}. 
BL was supported by National Key R\&D Program of
China (No. 2024YFA1013700), NSFC 12271176 and Guangdong Natural Science Foundation 2024A1515010946. R. Li was supported by the NSFC 12401006.  Y. F. Wu  (corresponding author) was supported by the NSFC 12301110.

\end{document}